\documentclass[oneside, headsepline, 10pt, BCOR=1mm,DIV=18, titlepage=false]{scrartcl}

\usepackage{amsfonts}
\usepackage{amssymb}
\usepackage{amsthm}
\usepackage{amsmath, amsfonts, amsxtra}
\usepackage{verbatim} %
\usepackage[colorlinks=true,linkcolor=blue, linktoc=all, citecolor=blue]{hyperref}%

\newtheorem{theorem}{Theorem}[section]%
\newtheorem{lemma}[theorem]{Lemma}%
\newtheorem{proposition}[theorem]{Proposition}%
\newtheorem{remark}[theorem]{Remark}%









\usepackage{array} 
\newenvironment{ma}{\begin{array}{>{\displaystyle}r >{\displaystyle}c >{\displaystyle}l}}{\end{array}}%

\newcommand{\aleq}{\prec}
\newcommand{\ageq}{\succ}
\newcommand{\aeq}{\approx}

\newcommand{\subsubset}{\subset\subset}
\newcommand{\supsupset}{\supset\supset}

\newcommand{\N}{{\mathbb N}}
\newcommand{\R}{{\mathbb R}}

\renewcommand{\S}{{\mathbb S}}

\newcommand{\Rz}{{\mathcal{R}}} 

\newcommand{\fracm}[1]{{\frac{1}{#1}}}

\newcommand{\supp}{\operatorname{supp}}%
\newcommand{\dist}{\operatorname{dist}}%

\newcommand{\lap}{\Delta}

\newcommand{\laps}[1]{\lap^{\frac{#1}{2}}}
\newcommand{\lapms}[1]{\lap^{-\frac{#1}{2}}}

\newcommand{\lapa}{\lap^{\frac{\alpha}{2}}}
\newcommand{\lapma}{\lap^{-\frac{\alpha}{2}}}

\newcommand{\abs}[1]{{\left \vert #1 \right \vert}}%
\newcommand{\sabs}[1]{{\vert #1 \vert}}%
\newcommand{\brac}[1]{{\left ( #1 \right )}}%
\newcommand{\Vrac}[1]{{\left \Vert #1 \right \Vert }}%
\newcommand{\vrac}[1]{{\Vert #1 \Vert }}%

\newcommand{\ontop}[2]{{\genfrac{}{}{0pt}{}{#1}{#2}}}

\newcommand{\intl}{\int \limits}

\def\XXint#1#2#3{{\setbox0=\hbox{$#1{#2#3}{\int}$}%
     \vcenter{\hbox{$#2#3$}}\kern-.5\wd0}}%

\newcommand{\sref}[2]{#1.\ref{#2}}

\numberwithin{equation}{section}%

\title{n/p-harmonic maps:\\regularity for the sphere case}
\author{Francesca Da Lio\footnote{Department of Mathematics, ETH Z\"urich, R\"amistr. 101, 8092 Z\"urich, Switzerland}$\ $\footnote{Dipartimento di Matematica Pura ed Applicata, Universit\`a degli Studi di Padova. Via Trieste 63, 35121,Padova, Italy, e-mail: dalio@mah.unipd.it}, Armin Schikorra\footnote{MPI MIS Leipzig, Inselstr. 22, 04315 Leipzig, Germany, e-mail: schikorr@mis.mpg.de; supported by the VARIOGEO-project by Prof. Jost, ERC Advanced Investigator Grant ERC-2010-AdG-20100224, Grant Agreement Number 267087. Partially supported by DAAD PostDoc Program (D/10/50763) and FIM at ETH Z\"urich}}
\date{\today}

 \usepackage{fancyhdr}

\begin{document}
\maketitle

\thispagestyle{empty}
\begin{abstract}
\noindent
We introduce $n$/$p$-harmonic maps as critical points of the energy
\[
 \mathcal{E}_{n,p}(v) = \intl_{\R^n} \abs{\laps{\alpha} v}^{p}
\]
where pointwise $v: D \subset \R^n \to \S^{N-1}$, for the $N$-sphere $\S^{N-1} \subset \R^N$ and $\alpha = \frac{n}{p}$. This energy combines the non-local behaviour of the fractional harmonic maps introduced by Rivi\`{e}re and the first author with the degenerate arguments of the $n$-laplacian. In this setting, we will prove H\"older continuity.
\end{abstract}

\tableofcontents

%
%

\newcommand{\pbar}{{\overline{p}}}
\newcommand{\abar}{{\overline{\alpha}}}
\newcommand{\etab}{{\overline{\eta}}}

\section{Introduction}
Our work is motivated by recent results \cite{DR1dSphere}, \cite{DR1dMan}, \cite{SNHarmS10}, \cite{DndMan}, \cite{Sfracenergy} which proved regularity for critical points of the energy $\mathcal{F}_n$ acting on maps $v: \R^n \to \R^N$,
\[
 \mathcal{F}_{n}(v) = \intl_{\R^n} \abs{\lap^{\frac{n}{4}} v}^2 \quad \mbox{$v \in \mathcal{N} \subset \R^N$ a.e.}
\]
Here, the operator $\laps{\alpha} v$ is defined as a multiplier operator with symbol $-\abs{\xi}^\alpha$, that is, denoting the Fourier transform and its inverse by $()^\wedge$ and $()^\vee$, respectively,
\[
 \lapa v = \brac{-\abs{\xi}^\alpha v^\wedge}^\vee.
\]
These energies were introduced by T. Rivi\`{e}re and the first author -- and they can be seen as an $n$-dimensional alternative to the two-dimensional Dirichlet energy
\[
\mathcal{D}_{2}(v) = \intl_{\R^2} \abs{\nabla v}^2 \quad \mbox{$v \in \mathcal{N} \subset \R^N$ a.e.}
\]
Both energies have critical Euler-Lagrange equations. That is, the highest order terms scale exactly as the lower-order terms, thus inhibiting the application of a general regularity theory based only on the general growth of the right-hand side -- one has to consider the finer behavior of the equation: These exhibit an antisymmetric structure, which is closely related to the appearance of Hardy spaces and compensated compactness -- and induces regularity of critical points. In two dimensions, these facts were observed in Rivi\`{e}re's celebrated \cite{Riv06} for all conformally invariant variational functionals (of which the Dirichlet energy is a prototype). We refer the interested reader to the introductions of \cite{DR1dSphere}, \cite{DR1dMan} for more on this.\\
Another possibility of generalizing the Dirichlet energy to arbitrary dimensions (whilst preserving the criticality of the Euler-Lagrange equations) is to consider
\[
 \mathcal{D}_{n}(v) = \intl_{\R^n} \abs{\nabla v}^n \quad \mbox{$v \in \mathcal{N} \subset \R^N$ a.e.}.
\]
Again in this case, the now degenerate Euler-Lagrange equations are critical and exhibit an antisymmetric structure, cf. \cite[Chapter III]{Riv08} -- only that it is not known so far, whether in general this structure implies even continuity. In fact, towards regularity of these systems, only few results are known. In \cite{Strz94} P. Strzelecki proved regularity, if the target manifold is a round sphere $\S^{N-1}$ -- which extended the respective Dirichlet-energy result by F. H\'elein \cite{Hel90}. In the setting of general manifolds, we know so far of convergence results, cf. \cite{WangCompThm05}, and only under additional assumptions on the solution there are regularity results, cf. \cite{DM10}, \cite{Kol10}, \cite{SnHsystemOrlicz}.\\
It then seems interesting to consider an energy which combines the difficulties of $\mathcal{D}_n$ and $\mathcal{F}_n$. Namely we will work with
\begin{equation}\label{eq:Enpdef}
 \mathcal{E}_{n,\pbar}(v) = \intl_{\R^n} \abs{\laps{\abar} v}^{\pbar} \quad \mbox{$v\big \vert_D \in \mathcal{N} \subset \R^N$ a.e., where $\abar = \frac{n}{\pbar}$, $D \subsubset \R^n$}
\end{equation}
Note, in the Euler-Lagrange equations of $\mathcal{E}_{n,\pbar}$, the leading order differential operator is nonlocal \emph{and} degenerate. Again, these settings are critical for regularity: One checks that any mapping $v$ with finite energy $\mathcal{E}_{n,\pbar}(v) < \infty$ belongs to BMO, but does not necessarily need to be continuous, as Frehse's counterexample \cite{Frehse73} shows.\\
Here we consider the situation of a sphere, i.e. $\mathcal{N} = \S^{N-1}$. Our main result is:
\begin{theorem}\label{th:main}
Let $u$ be a critical point of $\mathcal{E}_{n,\pbar}$ as in \eqref{eq:Enpdef}. Then $u$ is H\"older-continuous.
\end{theorem}
Naturally, one expects this result to hold at least partially for more general manifolds $\mathcal{N}$. To this end, in \cite{DSpalphMAN} we will treat this case of general manifolds, but with the condition $\pbar \leq 2$.\\
The proof relies on a suitable adaption of the arguments in \cite{DR1dSphere}, \cite{DR1dMan}, \cite{SNHarmS10}, \cite{DndMan}, \cite{Sfracenergy}, the details of which we will explain in the next section: The Euler-Lagrange equations of a critical point, see \cite{DR1dSphere},
\cite{SNHarmS10}, imply that
\begin{equation}\label{eq:upde}
\intl_{\R^n} \sabs{\laps{\abar} u}^{\pbar -2}\ \laps{\abar} u^i\ \laps{\abar}
(\omega_{ij} u^j \varphi) = 0 \quad \mbox{for all $\varphi \in C_0^\infty(D)$,
$\omega_{ij} = - \omega_{ji} \in \R$}.
\end{equation}
Note, that the main difference and difficulty comparing this equation to the $n/2$-harmonic case in \cite{DR1dSphere}, \cite{SNHarmS10}, is the weight $\sabs{\laps{\abar} u}^{\pbar -2}$! Moreover, we have the sphere-condition,
\begin{equation}\label{eq:uinsphere}
\abs{u(x)} = 1  \quad \mbox{for a.e. $x \in D$}.
\end{equation}
For a sketch of the proof, let us assume that $D = \R^n$. Note, that \eqref{eq:uinsphere} reveals information about the growth of derivatives of $u$ in the direction of $u$:
\[
 u \cdot \nabla u \equiv 0
\]
Moreover -- and more suitable to our case -- 
\begin{equation}\label{eq:threecomm}
 - 2 u \cdot \laps{\abar} u = \brac{\laps{\abar} \abs{u}^2 - u \cdot \laps{\abar} u - u \cdot \laps{\abar} u}\ - \underbrace{\laps{\abar} \overbrace{\abs{u}^2}^{\equiv 1}}_{\equiv 0}.
\end{equation}
We set 
\[
 H_{\alpha} (u,v) := \laps{\abar} (uv) - u \laps{\abar} v - v \laps{\abar} u,
\]
We will see that    $3$-term commutators $H_{\alpha} (u,v)$, appearing on the right-hand side of (\ref{eq:threecomm}),  are  more regular than each of their three generating terms. Compensation phenomena for 3-term commutators were first  observed by the first author and  Rivi\`ere   in \cite{DR1dSphere} and \cite {DR1dMan}  in the context of half-harmonic maps by using the so-called para-products.
Such compensation phenomena can be formulated in  different ways, for instance as an expansion of lower order derivatives.
This can best be seen by taking $\alpha = 2$: $H_2(u,v) = 2\nabla u \cdot \nabla v$ -- they behave like products of lower-order operators applied to $u$ and $v$. This interpretation has been developed by the second author in  \cite{Sfracenergy}, \cite{SNHarmS10} and it  is the approach that we will use in this paper.
We finally mention that these lower order expansions   are also closely related to the T1-Theorem and the ``Leibniz rule'' for fractional order derivatives obtained by Kato and Ponce, see \cite{KP88} and \cite[Corollary 1.2]{Hof98}. 
The necessary estimates for  the  operators $H_{\alpha} (u,v)$, can be paraphrased by

\begin{theorem}[cf. \cite{Sfracenergy}]\label{th:lowerorderalphaln}
For any $\alpha \in (0,n)$, 
Let $u =\lapms{\alpha} \laps{\alpha} u$, $v =\lapms{\alpha} \laps{\alpha} v$.
Then for $\alpha \in (0,n)$ there exists some constant $C_\alpha > 0$ and a
number $L \equiv L_\alpha \in \N$, and for $k \in \{1,\ldots,L\}$ constants $s_k
\in (0,\alpha)$, $t_k \in [0,s_k]$ such that for any $i = 1, \ldots,n$, where $\Rz_i$ denotes the Riesz transform,
\[
 \abs{\Rz_i H_\alpha(u,v)(x)} \leq C\ \sum_{k=1}^L M_k \lapms{s_k-t_k}
\brac{\lapms{t_k} \abs{\laps{\alpha} u}\ N_k
\lap^{-\frac{\alpha}{2}+\frac{s_k}{2}} \abs{\laps{\alpha} v}}.
\]
Here, $M_k, N_k$ are possibly Riesz transforms, or the identity. Moreover,
$\abs{s_k-t_k}$ can be supposed to be arbitrarily small. In particular,
for any $\alpha \in (0,n)$, $q, q_1, q_2 \in [1,\infty]$ such that
\[
 \fracm{q} = \fracm{q_1} + \fracm{q_2}.
\]
Then
\begin{equation}\label{eq:Halphauvlpqest}
 \vrac{H_\alpha (u,v)}_{(\frac{n}{\alpha},q),\R^n} \aleq \vrac{\lapa
u}_{(\frac{n}{\alpha},q_2),\R^n}\ \vrac{\lapa v}_{(\frac{n}{\alpha},q_2),\R^n}.
\end{equation}
Here, $\vrac{\cdot}_{(p,q)}$ denotes the Lorentz-space $L^{p,q}(\R^n)$-norm.
\end{theorem}
Consequently, \eqref{eq:uinsphere} controls $u\cdot \laps{\abar}u$ roughly like
\[
 \vrac{u\cdot \laps{\abar}u}_{\pbar,\R^n} \aleq \vrac{\laps{\abar}u}_{\pbar,\R^n}^2.
\]
This argument can be localized, and then implies an estimate for the growth $u\cdot \laps{\abar}u$ in the $L^{\pbar}$-norm on small balls by the square $\vrac{\laps{\abar}u}_{\pbar}^2$ localized essentially to slightly bigger balls.\\
Now the fact that $\abs{u} \equiv 1$, implies also that in order to control the growth of $\laps{\abar}u$, it suffices to estimate the growth of $u\cdot \laps{\abar}u$ and the growth of $\omega_{ij} u^i \laps{\abar} u^j$ for finitely many $\omega_{ij} = - \omega_{ji} \in \R$, see Proposition \ref{pr:orthogdecomp}. But terms of the form $\omega_{ij} u^i \laps{\abar} u^j$ can be estimated by the Euler-Lagrange equation \eqref{eq:upde}.\\
By this kind of argument, we obtain (essentially) the following growth estimates for all balls $B_r$
\[
 \vrac{\laps{\abar}u}_{\pbar,B_r} \leq \vrac{\laps{\abar}u}^2_{\pbar,B_{\Lambda r}} + \Lambda^{-\gamma}\ \vrac{\laps{\abar}u}_{\pbar,\R^n}\ \sum_{k=1}^\infty 2^{-\gamma k} \vrac{\laps{\abar}u}_{\pbar,B_{2^k \Lambda r}\backslash B_{2^k \Lambda r}},
\]
for some $\gamma > 0$, and any $\Lambda > 2$. Using an iteration technique, this implies that 
\[
 \vrac{\laps{\abar}u}_{\pbar,B_r} \leq C r^\alpha,
\]
which accounts for the H\"older-continuity of $u$.\\
\\
Let us briefly underline the differences to the manifold case treated in \cite{DSpalphMAN}. There the simple condition \eqref{eq:uinsphere} does not hold anymore and  we follow the approach introduced in \cite {DR1dMan}  which consists in considering  separately  the tangential and normal projections of $\laps{\abar} u$ (that is, we work with projections related to the \emph{derivatives} of $u$).  For the moment, this prevents us to treat the case of    {\em extremely small } $\abar$ (which in the sphere case poses no problems). On the other hand, the respective Euler-Lagrange equations actually exhibit a non-trivial right-hand side with antisymmetric structure. This will force us, to estimate the growth of $\laps{\abar} u$ in the weak space $L^{\pbar,\infty}$, which in turn will make it necessary to gain $L^{\pbar,1}$-estimates from the three-term commutators $H_{\abar}(\cdot,\cdot)$. This again, cf. \eqref{eq:Halphauvlpqest} for $q_1 = q_2 = 2$, will only be possible if $\pbar \leq 2$.

${}$\\[2em]
We will use \emph{notation} similar to \cite{SNHarmS10}:\\
We say that $A \subsubset \R^n$ if $A$ is a bounded subset of $\R^n$. For a set $A \subset \R^n$ we will denote its $n$-dimensional Lebesgue measure by $\abs{A}$. By $B_r(x) \subset \R^n$ we denote the open ball with radius $r$ and center $x \in \R^n$. If no confusion arises, we will abbreviate $B_r \equiv B_r(x)$. If $p \in [1,\infty]$ we usually will denote by $p'$ the H\"older conjugate, that is $\frac{1}{p} + \frac{1}{p'} = 1$. By $f \ast g$ we denote the convolution of two functions $f$ and $g$. Lastly, our constants -- frequently denoted by $C$ or $c$ --  can possibly change from line to line and usually depend on the space dimensions involved, further dependencies will be denoted by a subscript, though we will make no effort to pin down the exact value of those constants. If we consider the constant factors to be irrelevant with respect to the mathematical argument, for the sake of simplicity we will omit them in the calculations, writing $\aleq{}$, $\ageq{}$, $\aeq{}$ instead of $\leq$, $\geq$ and $=$.\\
We will use the same cutoff-functions as in, e.g., \cite{DR1dSphere}, \cite{SNHarmS10}: $\eta^k_{r} \in C_0^\infty(A_{r,k})$ where
\[
 B_{r,k}(x) := B_{2^kr}(x)
\]
for $k \geq 1$,
\[
 A_{r,k}(x) := B_{r,k+1}(x) \backslash B_{r,k-1}(x),
\]
and for $k = 0$
\[
 A_{r,0}(x) := B_{r,0}(x).
\]
Moreover, $\sum_k \eta^k_r \equiv 1$ pointwise everywhere, and we assume that $\abs{\nabla^l \eta^k_r} \leq C_l\  \brac{2^k r}^{-l}$.

\newpage
\section{Proof of Theorem \ref{th:main}}
Let $\abar \in (0,n)$, $\pbar  = \frac{n}{\abar} \in (1,\infty)$, and $u \in
L^{\pbar}(\R^n,\R^N)$, $\laps{\abar} u \in L^{\pbar }(\R^n,\R^N)$. Assume
moreover, that $D \subsubset \R^n$ such that \eqref{eq:uinsphere} \eqref{eq:upde} holds.

As  \eqref{eq:upde} and \eqref{eq:uinsphere} are equations satisfied by any
critical point $u$ of Theorem \ref{th:main}, we have to show the following 
\begin{theorem}\label{th:main2}
Let $u \in L^{\pbar}(\R^n)$, $\laps{\abar} u \in L^p(\R^n)$ satisfy
\eqref{eq:upde}, \eqref{eq:uinsphere}. Then $u$ is H\"older continuous in $D$.
\end{theorem}

In order to prove Theorem \ref{th:main2}, we first rewrite equations
\eqref{eq:upde} and \eqref{eq:uinsphere} in a fashion similar to
\cite{DR1dSphere},\cite{SNHarmS10}: Firstly, Equation \eqref{eq:upde} is
equivalent to
\begin{equation}\label{eq:uel}
\intl_{\R^n} \sabs{\laps{\abar} u}^{\pbar -2}\ u^j \omega_{ij} \laps{\abar} u^i\
\laps{\abar} \varphi = -\intl_{\R^n} \sabs{\laps{\abar} u}^{\pbar -2}\
\laps{\abar} u^i\ \omega_{ij}\ H(u^j,\varphi)
\quad \mbox{for all $\varphi \in C_0^\infty(D)$, $\omega_{ij} = - \omega_{ji}
\in \R$}.
\end{equation}
Here and henceforth, 
\[
  H(a,b) \equiv H_\abar(a,b) \equiv \laps{\abar}(ab) - a \laps{\abar} b - b
\laps{\abar} a.
\]
Assume we prove H\"older-continuity of $u$ in a Ball $B \subsubset D$. Pick a
slightly bigger ball $\tilde{B} \subsubset D$, $\tilde{B} \supsupset B$, and let
$w := \etab u$, for some 
\[
\etab \in C_0^\infty(D,[0,1]),\quad \etab \equiv 1 \mbox{ on
$\tilde{B}$}.
\]
Note that $w \in L^p(\R^n)$ for any $p \in [1,\infty]$. It suffices
to show H\"older regularity for $w$. The relevant equations for $w$ stemming
from \eqref{eq:uinsphere} and \eqref{eq:uel} are then (again, cf.
\cite{DR1dSphere}, \cite{SNHarmS10})
\begin{equation}\label{eq:winsphere}
w \cdot \laps{\abar} w = \frac{1}{2} H(w,w) + \frac{1}{2} \laps{\abar} \etab^2 
\quad \mbox{a.e. in $\R^n$},
\end{equation}
and for all $\varphi \in C_0^\infty(D)$, $\omega_{ij} = - \omega_{ji} \in \R$,
\begin{align}
&\intl_{\R^n} \sabs{\laps{\abar} w}^{\pbar -2}\ w^j \omega_{ij} \laps{\abar}
w^i\ \laps{\abar} \varphi\label{eq:wel}\\
&= \omega_{ij} \intl_{\R^n} \brac{\sabs{\laps{\abar} w}^{\pbar -2}\ \laps{\abar}
w^i\ -\sabs{\laps{\abar} u}^{\pbar -2} \laps{\abar} u^i} w^j  \laps{\abar}
\varphi \label{eq:wel:diff1}\\
&+ \omega_{ij} \intl_{\R^n} \brac{\sabs{\laps{\abar} u}^{\pbar -2} \laps{\abar}
u^i} (w^j-u^j)  \laps{\abar} \varphi\label{eq:wel:diff2}\\
& +\omega_{ij} \intl_{\R^n} \sabs{\laps{\abar} u}^{\pbar -2}\ \laps{\abar} u^i\
\ H(w^j-u^j,\varphi)\label{eq:wel:diff3}\\
  & +\omega_{ij} \intl_{\R^n} \brac{\sabs{\laps{\abar} w}^{\pbar -2}\
\laps{\abar} w^i - \sabs{\laps{\abar} u}^{\pbar -2}\ \laps{\abar} u^i}\ \
H(w^j,\varphi)\label{eq:wel:diff4}\\
& -\omega_{ij}\intl_{\R^n} \sabs{\laps{\abar} w}^{\pbar -2}\ \laps{\abar} w^i\ \
H(w^j,\varphi).\label{eq:wel:ess}
\end{align}
Now we need to appropriately adapt several arguments of
\cite{DR1dSphere},\cite{SNHarmS10}: First of all, using \eqref{eq:winsphere} we
control $\laps{\abar} w$ projected into the orthogonal space to the sphere at
the point $w$, $T^\perp_w \S^{N-1}$.
\subsection*{The orthogonal part}
 Namely, from \eqref{eq:winsphere} and
Lemma~\ref{la:Hvwlocest} one infers
\begin{lemma}\label{la:orthest}
There is $\gamma = \gamma_{\abar.\pbar} > 0$ and a constant $C$ depending on the
choice of $B$, $\tilde{B}$, $\etab$, such that the following holds: For any
$\varepsilon > 0$ there exists $\Lambda > 0$, $R > 0$, such that for any
$B_{\Lambda r} \subsubset B$, $r \in (0,R)$,
\[
 \vrac{w \cdot \laps{\abar} w}_{\pbar,B_r} \leq \varepsilon\ \vrac{ \laps{\abar}
w }_{\pbar, B_{\Lambda r}} + C\ r^\gamma + \varepsilon \sum_{k=1}^\infty
2^{-\gamma k}\ \Vert \laps{\abar} w \Vert_{\pbar,B_{2^k \Lambda r}\backslash
B_{2^{k-1} \Lambda r}}.
\]
\end{lemma}

The next step is to control the tangential part of $\laps{\abar} w$ by means of
\eqref{eq:wel}. The terms on the right-hand side of \eqref{eq:wel} can be
divided into two groups. The integrands of \eqref{eq:wel:diff1},
\eqref{eq:wel:diff2}, \eqref{eq:wel:diff3}, and \eqref{eq:wel:diff4} always
contain differences of the form $w-u$ which is trivial in $\tilde{B}$.
Consequently, we show that these terms behave subcritical.
\subsection*{Estimates of Tangential Part: Subcritical Terms
(\ref{eq:wel:diff1}), (\ref{eq:wel:diff2}), (\ref{eq:wel:diff3}),
(\ref{eq:wel:diff4})}
To be more precise, assume $\abs{\omega_{ij}} \leq 2$. We claim that for any
$\varphi \in C_0^\infty(B_r)$, where $B_{2r} \subset B$, all but the last term
on the right-hand side of \eqref{eq:wel} can be estimated by a constant
depending on $u$ and $B$ and the distance between $B_r$ and $\partial \tilde{B}$
times $r^\gamma$ for some $\gamma$:

\begin{proposition}[Subcritical Terms]\label{pr:differenceguys} There exists a
constant $C$ depending on $u$, $B$, $\tilde{B}$, the choice of $\etab$, and an
exponent $\gamma \equiv \gamma_{\pbar,\abar}> 0$ such that for any $\varphi \in
C_0^\infty(B_r)$, $\vrac{\laps{\abar} \varphi}_{\pbar} \leq 1$ for arbitrary
$B_{2r} \subset B$, if
\[
\begin{ma}
I &:=& \intl_{\R^n} \brac{\sabs{\laps{\abar} w}^{\pbar -2}\ \laps{\abar} w^i\
-\sabs{\laps{\abar} u}^{\pbar -2} \laps{\abar} u^i} w^j  \laps{\abar} \varphi,\\
II &:=& \intl_{\R^n} \brac{\sabs{\laps{\abar} u}^{\pbar -2} \laps{\abar} u^i}
(w^j-u^j)  \laps{\abar} \varphi,\\
III &:=& \intl_{\R^n} \sabs{\laps{\abar} u}^{\pbar -2}\ \laps{\abar} u^i\ \
H(w^j-u^j,\varphi),\\
IV &:=& \intl_{\R^n} \brac{\sabs{\laps{\abar} w}^{\pbar -2}\ \laps{\abar} w^i -
\sabs{\laps{\abar} u}^{\pbar -2}\ \laps{\abar} u^i}\ \ H(w^j,\varphi),
\end{ma}
\]
then
\[
 \abs{I} + \abs{II} + \abs{III} + \abs{IV} \leq C\ r^\gamma.
\]
\end{proposition}
\begin{proof}
Note that for any $x,y \in \R^n$ and any $a,b \in \R^N$ we have (cf., e.g.,
\cite[Proposition 4.1]{SNHarmS10Arxiv})
\[
 \abs{ \abs{a}^{p-2}a - \abs{b}^{p-2}b} \leq C_p\ \begin{cases}
				      \abs{a-b}^{p-1} \quad &\mbox{if $p \in
[1,2]$,}\\
                                      \abs{a-b}^{p-1} + \abs{a-b} \abs{b}^{p-2}
\quad &\mbox{if $p > 2$}.
                                     \end{cases}
\]
Now we argue via Proposition~\ref{pr:stupidestumw}. In order to do so, we use that there is some
positive distance $d \equiv d_{B,\tilde{B}} > 0$, such that $\dist (\supp \varphi,
\R^n \backslash \tilde{B}) > d$. This is straight-forward for $I$ and $IV$
(recall that $w \in L^\infty$ for $I$ and use Proposition~\ref{pr:Hest} for
$IV$). For $III$, we apply Proposition~\ref{pr:stupidestumw} to the terms,
\[
 H(w^j-u^j,\varphi) = H((\etab-1)u^j,\varphi) = \varphi\ \laps{\abar}
((\etab-1)u^j) + (\etab-1)u^j\ \laps{\abar} \varphi,
\]
and finally, for $II$ we first estimate
\[
 \vrac{(1-\etab) u \laps{\abar} \varphi }_{\pbar} \aleq \Vrac{u}_{\pbar}\ \
\Vert (1-\etab)\laps{\abar} \varphi \Vert_{\infty}.
\]
\end{proof}

In the proof of Proposition~\ref{pr:differenceguys} we used the following
estimate, which can be proven by an argument which appears in a similar form
already in \cite{DR1dSphere}.
\begin{proposition}[Estimates for disjoint-support terms]\label{pr:stupidestumw}
Let $r \in (0,1)$, $d > 0$, $p, q \in [1,\infty)$ such that $B_{r+d} \subset
\tilde{B}$. Then, for any $f \in L^q(\R^n)$, with $\Delta^{\alpha/2} f \in L^q(\R^n)$, 
\begin{equation}\label{eq:prstupidestumw:first}
 \Vrac{ \laps{\alpha} ((1-\etab)f)}_{p,B_{r}} \aleq C_d\ r^{\frac{n}{p}}\
({\vrac{f}_{q,\R^n}+\vrac{\lapa f}_{q,\R^n}}).
\end{equation}
And if moreover $\supp f \subset B_r$, for some $\gamma = \gamma_\alpha$,
\begin{equation}\label{eq:prstupidestumw:second}
 \Vrac{(1-\etab) \laps{\alpha} f}_{\infty,\R^n} \aleq C_d\ r^{\gamma}\
\Vrac{\lapa f}_{q}.
\end{equation}
\end{proposition}
\begin{proof}
Instead of the argument using Fourier-transform as in \cite{DR1dSphere},
\cite{SNHarmS10} etc., we use the following argument: Let $\alpha =: K/2 + s$,
where $s \in (0,2)$ (the case $s = 0$ is trivial), $K \in 2\N$, that is
\[
 \lapa = \laps{s} \lap^K.
\]
Set $h := \lap^K g$, and recall that
\[
 \laps{s} h(x) = c \int \frac{h(x+z)+h(x-z) - 2 h(x)}{\abs{z}^{n+s}}\ dz
\]
Thus, if $\supp g \subset \R^n \backslash B_{r+d}$ and $x \in B_r$, as is the
case in \eqref{eq:prstupidestumw:first}, or $\supp g \subset B_r$ and $x \in
\R^n \backslash B_{r+d}$, as is the case in \eqref{eq:prstupidestumw:second} 
\[
 \sabs{\laps{s} h(x)} \aleq \int_{\abs{z} > d} \abs{h(x+z)}\ \abs{z}^{-n-s}\ dz
= \abs{h}\ast \brac{\abs{\cdot}^{-n-s}\chi_{\abs{\cdot}>d}}(x),
\]
and
\[
 \begin{ma}
  \vrac{\abs{h}\ast \brac{\abs{\cdot}^{-n-s}\chi_{\abs{\cdot}>d}}}_{p} &\aleq&
\vrac{h}_{q}\ d^{-n(q_2-1)-2s q_2}\\
&\aleq& d^{-n(q_2-1)-2s q_2}\ ({\vrac{f}_{q,\R^n}+\vrac{\lap^K f}_{q,\R^n}}),
\end{ma}
\]
where $q_2$ is chosen such that
\[
 1+\fracm{p} = \fracm{q} + \fracm{q_2}.
\]
For \eqref{eq:prstupidestumw:first} we then use the fact that $W^{K,q} \subset
W^{\alpha,q}$, that is
\[
 \vrac{f}_{q,\R^n}+\vrac{\lap^K f}_{q,\R^n} \aleq \vrac{f}_{q,\R^n}+\vrac{\lapa
f}_{q,\R^n}.
\]
For \eqref{eq:prstupidestumw:first}, note that by Poincar\'e's inequality,
\[
 \vrac{f}_{q,\R^n}+\vrac{\lap^K f}_{q,\R^n} \aleq \vrac{\lap^K f}_{q,\R^n} \aleq
r^{\alpha-K}\ \vrac{\lapa f}_{q,\R^n}.
\]

\end{proof}

\subsection*{Estimates of Tangential Part - The Critical Term (\ref{eq:wel:ess})}
In order to estimate the tangential part of $\laps{\abar} w$ completely, we need to control the last term \eqref{eq:wel:ess}, which is done in the following
two propositions.

\begin{proposition}\label{pr:Hwvpakest}
There is a constant $C_\abar$, $R > 0$ and an exponent $\gamma > 0$ such that
the following holds. Let $\Vert w\Vert_{\infty} \leq 1$, and assume
$\laps{\abar} w \in L^{\pbar}(\R^n)$. Then for any $\varphi \in
C_0^\infty(B_r)$, $\Lambda \geq 5$, $r \in (0,R)$
\[
 \Vrac{H(w,\varphi)}_{\pbar,B_{2^k \Lambda r} \backslash B_{2^{k-1} \Lambda r}}
\leq C\ \ (2^k \Lambda)^{-\gamma}\ \Vert \laps{\abar} \varphi \Vert_{\pbar}.
\]
\end{proposition}
\begin{proof}
We have on $B_{2^k \Lambda r} \backslash B_{2^{k-1} \Lambda r}$
\[
 H(w,\varphi) = \laps{\abar} (w\varphi) - w \laps{\abar} \varphi.
\]
One checks that (exploiting the disjoint support via similar arguments as in
Proposition~\ref{pr:stupidestumw})
\[
 \Vert \laps{\abar} (w\varphi) \Vert_{\pbar,B_{2^k \Lambda r} \backslash
B_{2^{k-1} \Lambda r}} \aleq  \Vert w \Vert_{\infty}\ \brac{2^k \Lambda}^{-n}.
\]
and as well
\[
 \Vert w\ \laps{\abar} \varphi \Vert_{\pbar,B_{2^k \Lambda r} \backslash
B_{2^{k-1} \Lambda r}} \aleq  \Vert w \Vert_{\infty}\ \brac{2^k \Lambda}^{-n}.
\]
\end{proof}
Moreover,
\begin{proposition}\label{pr:HwvpBlrest}
There is a constant $C_\abar$ and an exponent $\gamma > 0$ such that the
following holds. Let $\Vert w\Vert_{\infty} \leq 1$, and assume $\laps{\abar} w
\in L^{\pbar}(\R^n)$. Then for any $\varphi \in C_0^\infty(B_r)$, $\Lambda \geq
5$
\[
 \Vrac{H(w,\varphi)}_{\pbar,B_{\Lambda r}} \leq C\
\brac{\Lambda^{-\gamma}\ \Vert \laps{\abar} w \Vert_{\pbar}+\Vert
\eta_{\Lambda^3 r}\laps{\abar} w \Vert_{\pbar}} \Vert \laps{\abar} \varphi
\Vert_{\pbar}.
\]
\end{proposition}
\begin{proof}
Set
\[
 w = \lapma (\eta_{\Lambda^3 r} \laps{\abar} w) + \lapma ((1-\eta_{\Lambda^3 r})
\laps{\abar} w) =: w_1 + w_2,
\]
and
\[
 \varphi = \lapma (\eta_{\Lambda r}\laps{\abar} \varphi) + \lapma
((1-\eta_{\Lambda r}) \laps{\abar} \varphi) =: \varphi_1 + \varphi_2,
\]
Then
\[
 H(w,\varphi) = H(w_1,\varphi) + H(w_2,\varphi_2) + H(w_2,\varphi_1).
\]
We compute via Proposition~\ref{pr:Hest},
\[
 \Vrac{H(w_1,\varphi)}_{\pbar} \aleq \Vert \eta_{\Lambda^3r} \laps{\abar} w
\Vert_{\pbar}\  \Vert \laps{\abar} \varphi \Vert_{\pbar},
\]
\[
 \Vrac{H(w_2,\varphi_2)}_{\pbar} \aleq \Vert \laps{\abar} w \Vert_{\pbar}\ 
\Vrac {\brac{1-\eta_{\Lambda r}} \laps{\abar} \varphi}_{\pbar} \aleq
\Lambda^{-n}\ \Vert \laps{\abar} w \Vert_{\pbar}\  \Vrac{ \laps{\abar}
\varphi}_{\pbar}.
\]
In order to estimate $\Vert H(w_2,\varphi_1) \Vert_{\pbar}$ in the given set,
according to Lemma~\ref{la:lowerorderalphaln}, similar to the arguments in
\cite{Sfracenergy}, it suffices to estimate terms of the following form, for
some $\psi \in C_0^\infty(B_r)$, $\Vert \psi \Vert_{\pbar'} \leq 1$, and for $s
\in (0,\abar)$, $t \in [0,s]$ (in fact, there might appear additional
$0$-multipliers, but as they do not interfere with the argument, we ignore this
for the sake of readability)
\[
 \begin{ma}
 &&\int \lapms{s-t} \psi\ \lapms{s} \abs{(1-\eta_{\Lambda^3 r}) \laps{\abar} w}\
\lapms{\abar-t} \abs{\eta_{\Lambda r} \laps{\abar} \varphi} \\
&=&\int \lapms{s-t} \psi\ \eta_{\Lambda^2 r}\lapms{t} \abs{(1-\eta_{\Lambda^3
r}) \laps{\abar} w}\ \lapms{\abar-s} \abs{\eta_{\Lambda r} \laps{\abar} \varphi}
\\
&&+\int (1-\eta_{\Lambda^2 r})\lapms{s-t} \psi\ \lapms{t}
\abs{(1-\eta_{\Lambda^3 r}) \laps{\abar} w}\ \lapms{\abar-s} \abs{\eta_{\Lambda
r} \laps{\abar} \varphi} \\
&\aleq{}&
\Vrac{\eta_{\Lambda^2 r}\lapms{t} \abs{(1-\eta_{\Lambda^3 r}) \laps{\abar}
w}}_{\frac{n}{\abar-t}}\ \Vrac{\laps{\abar} \varphi}_{\pbar} \\
&&+\Vrac{(1-\eta_{\Lambda^2 r})\lapms{s-t} \psi}_{\frac{n}{n-s+t-\abar}}\
\Vrac{\laps{\abar} w}_{\pbar}\ \Vrac{\laps{\abar} \varphi}_{\pbar} \\
&\aleq{}&
\brac{\Lambda^{t-\abar}  +\Lambda^{-\abar}}\ \Vrac{\laps{\abar} w}_{\pbar}\
\Vrac{\laps{\abar} \varphi}_{\pbar}
\end{ma}
\]
Here, we have used several times the arguments for products of (non-local)
fractional operators with disjoint support, for the details of which we refer
to, e.g., the arguments of Proposition~\ref{pr:stupidestumw} or
\cite[Proposition 4.4]{Sfracenergy}.\\
Setting $\gamma := \min\{\abar-t,\abar,n\}$, we conclude.
\end{proof}

\subsection*{Conclusion for the tangential part}
By Proposition~\ref{pr:differenceguys}, Proposition~\ref{pr:Hwvpakest} and
Proposition~\ref{pr:HwvpBlrest}, we arrive at the following
\begin{proposition}[Estimate of (\ref{eq:wel:ess})]\label{pr:tangrhsest}
There is $\gamma = \gamma_{\abar.\pbar} > 0$ and a constant $C$ depending on the
choice of $B$, $\tilde{B}$, $\etab$, such that the following holds: For any
$\varepsilon > 0$ there exists $\Lambda > 0$, $R > 0$, such that for any $B_{r}
\subsubset B$, $r \in (0,R)$, $\varphi \in C_0^\infty(B_r)$, $\vrac{\laps{\abar}
\varphi}_{\pbar} \leq 1$, $\omega_{ij} = -\omega{ji}$, $\abs{\omega} \leq 1$,
\[
 \omega_{ij}\intl_{\R^n} \sabs{\laps{\abar} w}^{\pbar -2}\ w^j \laps{\abar} w^i\
\laps{\abar} \varphi
 \leq \varepsilon\ \vrac{ \laps{\abar} w }^{\pbar-1}_{\pbar, B_{\Lambda r}} + C\
r^\gamma + \varepsilon \sum_{k=1}^\infty 2^{-\gamma k}\ \Vert \laps{\abar} w
\Vert_{\pbar,B_{2^k \Lambda r}\backslash B_{2^{k-1} \Lambda r}}^{\pbar-1}.
\]
\end{proposition}
Now we can proceed by virtually the same arguments as in
\cite{DR1dSphere},\cite{SNHarmS10}: Firstly, Proposition~\ref{pr:tangrhsest}
finally implies
\begin{lemma}\label{la:sopartest}
There is $\gamma = \gamma_{\abar.\pbar} > 0$ and a constant $C$ depending on the
choice of $B$, $\tilde{B}$, $\etab$, such that the following holds: For any
$\varepsilon > 0$ there exists $\Lambda > 0$, $R > 0$, such that for any $B_{r}
\subsubset B$, $r \in (0,R)$, $\omega_{ij} = -\omega{ji}$, $\abs{\omega} \leq
1$,
\[
 \vrac{\sabs{\laps{\abar} w}^{\pbar -2}\ \omega_{ij}\ w^j \laps{\abar}
w^i}_{\pbar', B_{\Lambda^{-1}r}}
 \leq \varepsilon\ \vrac{ \laps{\abar} w }^{\pbar-1}_{\pbar, B_{\Lambda r}} + C\
r^\gamma + \varepsilon \sum_{k=1}^\infty 2^{-\gamma k}\ \Vert \laps{\abar} w
\Vert_{\pbar,B_{2^k \Lambda r}\backslash B_{2^{k-1} \Lambda r}}^{\pbar-1}.
\]
\end{lemma}
\subsection*{Putting tangential and normal part together}
Together, Lemma~\ref{la:sopartest} and Lemma~\ref{la:orthest} imply
\begin{lemma}\label{la:wallest}
There is $\gamma = \gamma_{\abar.\pbar} > 0$ and a constant $C$ depending on the
choice of $B$, $\tilde{B}$, $\etab$, such that the following holds: For any
$\varepsilon > 0$ there exists $\Lambda > 0$, $R > 0$, such that for any $B_{r}
\subsubset B$, $r \in (0,R)$,
\[
 \vrac{\laps{\abar} w^i}_{\pbar, B_{\Lambda^{-1}r}}^{\pbar-1}
 \leq \varepsilon\ \vrac{ \laps{\abar} w }^{\pbar-1}_{\pbar, B_{\Lambda r}} + C\
r^\gamma + \varepsilon \sum_{k=1}^\infty 2^{-\gamma k}\ \Vert \laps{\abar} w
\Vert_{\pbar,B_{2^k \Lambda r}\backslash B_{2^{k-1} \Lambda r}}^{\pbar-1}.
\]
\end{lemma}
\begin{proof}
Note that, if $\pbar \geq 2$,
\[
\vrac{\omega_{ij}\ w^j \laps{\abar} w^i}^{\pbar -1}_{\pbar, B_{\Lambda^{-1}r}}
= \vrac{\sabs{\omega_{ij}\ w^j \laps{\abar} w^i}^{\pbar -2}\ \omega_{ij}\ w^j
\laps{\abar} w^i}_{\pbar', B_{\Lambda^{-1}r}}
 \aleq \vrac{\sabs{\laps{\abar} w}^{\pbar -2}\ \omega_{ij}\ w^j \laps{\abar}
w^i}_{\pbar', B_{\Lambda^{-1}r}}
\]
On the other hand, if $\pbar \in (1,2)$, we set for $\theta > 0$ from
Proposition~\ref{pr:orthogdecomp},
\begin{equation}\label{eq:Awtheta}
 A_{\omega,\theta} := \left \{ \sabs{w^i \omega_{ij} \lapa w^j} \geq \theta\
\sabs{\laps{\abar} w} \right \}.
\end{equation}
Said Proposition~\ref{pr:orthogdecomp} then implies that for any $\omega \in
\Omega$,
\[
\begin{ma}
 &&\vrac{\omega_{ij}\ w^j \laps{\abar} w^i}^{\pbar -1}_{\pbar,
B_{\Lambda^{-1}r}} \\
&\aeq&  \vrac{\omega_{ij}\ w^j \laps{\abar} w^i}^{\pbar -1}_{\pbar,
B_{\Lambda^{-1}r}\cap A_{\omega,\theta}}  + \vrac{\omega_{ij}\ w^j \laps{\abar}
w^i}^{\pbar -1}_{\pbar, B_{\Lambda^{-1}r}\cap A_{\omega,\theta}^c}\\

&\overset{\sref{P}{pr:orthogdecomp}}{\aleq}&  \vrac{\omega_{ij}\ w^j
\laps{\abar} w^i}^{\pbar -1}_{\pbar, B_{\Lambda^{-1}r}\cap A_{\omega,\theta}} + 
\sum_{\tilde{\omega} \in \Omega} \vrac{\tilde{\omega}_{ij}\ w^j \laps{\abar}
w^i}^{\pbar -1}_{\pbar, B_{\Lambda^{-1}r}\cap A_{\tilde{\omega},\theta}} +
\vrac{w^i \laps{\abar} w^i}^{\pbar -1}_{\pbar, B_{\Lambda^{-1}r}}\\
&\overset{\eqref{eq:Awtheta}}{\aleq}&  C_\theta\ \sum_{\tilde{\omega} \in
\Omega} \vrac{\sabs{\laps{\abar} w}^{\pbar -2}\ \tilde{\omega}_{ij}\ w^j
\laps{\abar} w^i}_{\pbar', B_{\Lambda^{-1}r}}  + \vrac{w^i \laps{\abar}
w^i}^{\pbar -1}_{\pbar, B_{\Lambda^{-1}r}}.
\end{ma}
\]
Applying again Proposition~\ref{pr:orthogdecomp}, we arrive for any $\pbar > 1$
at
\[
 \vrac{\laps{\abar} w}^{\pbar-1}_{\pbar, B_{\Lambda^{-1}r}} \aleq \sum_{\omega
\in \Omega} \vrac{\sabs{\laps{\abar} w}^{\pbar -2}\ \omega_{ij}\ w^j
\laps{\abar} w^i}_{\pbar', B_{\Lambda^{-1}r}}  + \vrac{w^i \laps{\abar}
w^i}^{\pbar -1}_{\pbar, B_{\Lambda^{-1}r}}.
\]
We conclude by Lemma~\ref{la:sopartest} and Lemma~\ref{la:orthest}.
\end{proof}

The proof of Theorem \ref{th:main2} follows by an iteration argument and an
application of Dirichlet's growth theorem (cf.
\cite{DR1dSphere},\cite{Sfracenergy}).
\newpage
\renewcommand{\thesection}{A}
\renewcommand{\thesubsection}{A.\arabic{subsection}}
\section{Lower Order Arguments: Proof of Theorem \ref{th:lowerorderalphaln}}
Let $\alpha \in (0,n)$. In this section, we treat estimates on the bilinear
operator
\[
 H_\alpha(u,v) := \laps{\alpha} (uv) - u \laps{\alpha} v - v \laps{\alpha} u,
\]
which behaves, in some sense, like a product of lower operators (take, for
example, the classic case $\alpha = 2$). The estimates are similar to the ones
obtained and used in \cite{DR1dSphere},\cite{SNHarmS10}, \cite{DR1dMan},
\cite{DndMan}, \cite{Sfracenergy}, and here, we will adopt the general strategy
of the latter article \cite{Sfracenergy}. The argument relies on the estimates
on singular kernels of Proposition~\ref{pr:multest}. These kernels (the
geometric-space analogue to the Fourier-multiplier estimates in
\cite{SNHarmS10}) appear because of the following representation of $\lapa$,
$\lapma$: For $f \in C_0^\infty(\R^n)$,
\[
 \lapa f(x) = c_{\alpha,n}\ \intl_{\R^n} \frac{f(x)-f(y)}{\abs{x-y}^{n+\alpha}}\
dx,  \quad \mbox{if $\alpha \in (0,1)$}.
\]
The inverse operator of $\lapa$ is the so-called Riesz potential,
\begin{equation}\label{eq:lo:rieszpotentialdef}
 \lapma f(x) = c_{\alpha,n}\ \intl_{\R^n} f(\eta)\ \abs{x-\eta}^{-n+\alpha}\ d\eta, 
\quad \mbox{if $\alpha \in (0,n)$}.
\end{equation}
The essential argument appears in the case $\alpha \in (0,1)$. As mentioned above, they are similar to the ones in \cite{Sfracenergy}, though there, for convenience, only special cases for $\alpha$ were considered.
\subsection{The case \texorpdfstring{$\alpha \in (0,1)$}{alpha in (0,1)}}
With the representation for $\laps{\alpha}$, one has for $\alpha \in (0,1)$,
\[
\begin{ma}
\laps{\alpha} (u\ v)(x)
&=&c_{\alpha,n} \int \frac{u(x)\ v(x) -
u(y)\ v(y)}{\abs{x-y}^{n+\alpha}}\ dy\\
&=&c_{\alpha,n} \int \frac{\brac{u(x)-u(y)}\
v(x) + u(y)\ \brac{v(x)-v(y)}}{\abs{x-y}^{n+\alpha}}\ dy\\
&=&\lapa u(x)\ v(x) + c_{\alpha,n} \int \frac{\brac{u(y)-u(x)}\ \brac{v(x)-v(y)}}{\abs{x-y}^{n+\alpha}}\ dy + \lapa b(x)\ u(x),
\end{ma}
\]
that is
\[
H_\alpha (u,v) =  c_{\alpha,n} \int \frac{\brac{u(y)-u(x)}\ \brac{v(x)-v(y)}}{\abs{x-y}^{n+\alpha}}\ dy.
\]
Replacing now $u := \lapms{\alpha} a$, $v := \lapms{\alpha} b$, this is equivalent to
\[
H_\alpha (u,v) =  \tilde{c}_{\alpha,n} \int \int \int \frac{
(\abs{y-\eta}^{-n+\alpha}-\abs{x-\eta}^{-n+\alpha})\ (\abs{y-\xi}^{-n+\alpha}-\abs{x-\xi}^{-n+\alpha})}{\abs{x-y}^{n+\alpha}}\ a(\eta)\ b(\xi)\ d\xi\ d\eta\ dy.
\]
Thus, the main point of our argument is to replace the differences of functions in the
definition of $H$ by differences on the kernels of the respective operators. In the above representation of $H_\alpha(u,v)$ it is useful to observe: whenever $\abs{x-y}^{-n-\alpha}$ becomes singular, both, the difference of terms with $\eta$ and the terms with $\xi$ tend to zero as well, thus ``absorbing'' the singularity up to a certain point. More precisely, these kernels can be estimated as in the following proposition. 
\begin{proposition}[Multiplier estimate]\label{pr:multest}
Let $\alpha \in [0,1]$ and $\varepsilon \in (0,1)$. Then for almost every
$x,y,\eta,\xi \in \R^n$, and a uniform constant $C$
\[
\begin{ma}
&&{\abs{\abs{\eta - y}^{-n+\alpha}-\abs{\eta - x}^{-n+\alpha}}\
\abs{\abs{\xi - y}^{-n+\alpha}-\abs{\xi -
x}^{-n+\alpha}}}\\
&\leq& C\ {\abs{y-\eta}^{-n+\alpha-\varepsilon}\
\brac{\abs{x-\xi}^{-n+\alpha-\varepsilon} +
\abs{y-\xi}^{-n+\alpha-\varepsilon}}}{\abs{x-y}^{2\varepsilon}}\\
&&+ C\ {\brac{\abs{x-\eta}^{-n+\alpha-\varepsilon} +
\abs{y-\eta}^{-n+\alpha-\varepsilon}}\
\abs{y-\xi}^{-n+\alpha-\varepsilon}}{\abs{x-y}^{2\varepsilon}}\\
&&+ C\ {\abs{x-\eta}^{-n+\alpha}
\abs{x-\xi}^{-n+\alpha}}\ \chi_{\abs{x-y} > 2
\abs{x-\xi}}\ \chi_{\abs{x-y} > 2 \abs{x-\eta}}.
\end{ma}
\]
\end{proposition}
In particular, multiplying this estimate with the hypersingular (i.e. not locally integrable) kernel $\abs{x-y}^{-n-\alpha}$ (the kernel of the differentiation $\laps{\alpha}$), choosing $\varepsilon$ such that $1 > 2\varepsilon > \alpha > \varepsilon > 0$, the first two terms on the right-hand side consists of products of locally integrable kernels, or, more precisely, kernels of the Riesz potentials $\lapms{\alpha - \varepsilon}$ and $\lapms{2\varepsilon -\alpha}$. In the last term, we then have twice kernels of $\lapms{\alpha}$, and the kernel $\abs{x-y}^{-n-\alpha}\chi_{\abs{x-y} > 2
\abs{x-\xi}}\ \chi_{\abs{x-y} > 2 \abs{x-\eta}}$ where the singularity $x=y$ is somewhat cut away.\\
As we will see in Proposition \ref{pr:lowerorderalphal1}, this enables us, to show that the operators $\lapa$ in the definition of $H_\alpha(\cdot,\cdot)$ ``distribute their differentiation'' on both entry-functions.
  
\begin{proof}[Proof of Proposition \ref{pr:multest}]
Observe that in the following argument, since $\alpha \in [0,1]$, all the
constants can be taken independently from $\alpha$.
We set
\[ 
 k(x,y,\eta) := \abs{\abs{\eta - y}^{-n+\alpha}-\abs{\eta - x}^{-n+\alpha}}.
\]
Decompose the space $(x,y,\eta) \in \R^{3n}$ into several subspaces depending on
the relations of $\abs{y-\eta}$, $\abs{x-y}$, $\abs{x-\eta}$:
\[
 1 \leq \chi_1(x,y,\eta) + \chi_2(x,y,\eta) + \chi_3(x,y,\eta) \quad \mbox{for
$x,y,\eta \in \R^n$},
\]
where
\[
 \chi_1 := \chi_{\abs{x-y} \leq  2\abs{y-\eta}}\ \chi_{\abs{x-y} \leq 2
\abs{x-\eta}},
\]
\[
 \chi_2 := \chi_{\abs{x-y} \leq  2\abs{y-\eta}}\ \chi_{\abs{x-y} > 2
\abs{x-\eta}},
\]
\[
 \chi_3 := \chi_{\abs{x-y} >  2\abs{y-\eta}}\ \chi_{\abs{x-y} \leq 2
\abs{x-\eta}},
\]
and functions of the form $\chi_{f(x,y,\eta) < 0}$ denote the usual
characteristic functions of the set $\{(x,y,\eta) \in \R^{3N}:\ f(x,y,\eta) <
0\}$. Note that with a uniform constant
\begin{equation}\label{eq:lo:ymxapxmxi}
 \abs{y-\eta} \chi_1 \approx \abs{x-\eta} \chi_1.
\end{equation}
Then, by the mean value theorem, for the details in this context cf.
\cite[Proposition 3.3]{Sfracenergy}, for any $\varepsilon \in (0,1)$
\[
 k(x,y,\eta) \chi_1 \aleq \abs{x-\eta}^{-n+\alpha -1} \abs{x-y}\chi_1 \aleq
\abs{x-\eta}^{-n+\alpha -\varepsilon} \abs{x-y}^\varepsilon\ \chi_1 \overset{\eqref{eq:lo:ymxapxmxi}}{\approx}
\abs{y-\eta}^{-n+\alpha -\varepsilon} \abs{x-y}^\varepsilon\ \chi_1.
\]
Moreover, for any $\varepsilon > 0$,
\[
 k(x,y,\eta) \chi_2 \aleq \abs{x-\eta}^{-n+\alpha} \chi_2 \aleq
\abs{x-\eta}^{-n+\alpha-\varepsilon}\ \abs{x-y}^{\varepsilon},
\]
and also for any $\varepsilon > 0$,
\[
 k(x,y,\eta) \chi_3 \aleq \abs{y-\eta}^{-n+\alpha} \chi_3 \aleq
\abs{y-\eta}^{-n+\alpha-\varepsilon}\ \abs{x-y}^{\varepsilon}.
\]
In order to estimate the product $k(x,y,\xi)k(x,y,\eta)$ we have to check the claim for all cases $(i,j)$, $i,j \in \{1,2,3\}$, where we say
\[
 \mbox{case }(i,j) \Leftrightarrow (x,y,\eta,\xi) \in \R^{4n} \mbox{ such that }\chi_i(x,y,\eta)\chi_j(x,y,\xi) =1.
\]
Lets denote
\[
\begin{ma}
\mbox{Type I-estimate} &:=& {\abs{y-\eta}^{-n+\alpha-\varepsilon}\
\abs{x-\xi}^{-n+\alpha-\varepsilon}} {\abs{x-y}^{2\varepsilon}}\\
\mbox{Type II-estimate} &:=& {\abs{y-\eta}^{-n+\alpha-\varepsilon}\
\abs{y-\xi}^{-n+\alpha-\varepsilon}} {\abs{x-y}^{2\varepsilon}}\\
\mbox{Type III-estimate} &:=& \abs{x-\eta}^{-n+\alpha-\varepsilon}\
\abs{y-\xi}^{-n+\alpha-\varepsilon}{\abs{x-y}^{2\varepsilon}}\\
\mbox{Type IV-estimate} &:=& {\abs{x-\eta}^{-n+\alpha}
\abs{x-\xi}^{-n+\alpha}}\ \chi_{\abs{x-y} > 2
\abs{x-\xi}}\ \chi_{\abs{x-y} > 2 \abs{x-\eta}}.
\end{ma}
\]
One checks, that each of these types have to appear. Note, the only case where we choose the Type IV-estimate is $(2,2)$.
\end{proof}
As a consequence of Proposition \ref{pr:multest}, we obtain the following estimate for $\alpha \in (0,1)$:
\begin{proposition}\label{pr:lowerorderalphal1}
Let $u =\lapms{\alpha} \laps{\alpha} u$, $v =\lapms{\alpha} \laps{\alpha} v$.
Then for $\alpha \in (0,1)$  and for  $\varepsilon \in (0,1)$ satisfying
\[
 \alpha < 2\varepsilon < \min\{2\alpha,n+\alpha\},
\]
there exists $C\equiv C_\alpha > 0$ such that
\[
\begin{ma}
 &&\abs{H(\lapms{\alpha}a,\lapms{\alpha}b)(x)}\\
&\leq& C [ \lapms{\varepsilon} \abs{\laps{\alpha} u}(x)\ \lapms{\alpha-\varepsilon}\abs{\laps{\alpha} v}(x) +
\lapms{\alpha-\varepsilon}\abs{\laps{\alpha} u}(x)\ \lapms{\varepsilon} \abs{\laps{\alpha} v}(x)\\
&&\lapms{2\varepsilon - \alpha}\brac{\lapms{\alpha-\varepsilon} \abs{\laps{\alpha} u}\
\lapms{\alpha-\varepsilon}\abs{\laps{\alpha} v}}(x)\\
&&+\lap^{-\frac{\alpha}{4}}\abs{\laps{\alpha} u}(x)\ \lap^{-\frac{\alpha}{4}}\abs{\laps{\alpha} v}(x)].
\end{ma}
\]
\end{proposition}

\begin{proof}
Set $a := \laps{\alpha} u$, $b := \laps{\alpha} v$. By the definition of $\lapa$
for $\alpha \in (0,1)$,
\[
\begin{ma}
 &&\laps{\alpha} (\lapms{\alpha} a\ \lapms{\alpha}b)(x)\\
&=&c_{n,\alpha} \int \frac{\lapms{\alpha} a(x)\ \lapms{\alpha} b(x) -
\lapms{\alpha} a(y)\ \lapms{\alpha} a(y)}{\abs{x-y}^{n+\alpha}}\ dy\\
&=&c_{n,\alpha} \int \frac{\brac{\lapms{\alpha} a(x)-\lapms{\alpha} a(y)}\
\lapms{\alpha} b(x) + \lapms{\alpha} a(y)\ \brac{\lapms{\alpha}
b(x)-\lapms{\alpha} b(y)}}{\abs{x-y}^{n+\alpha}}\ dy\\
&=&a(x)\ \lapms{\alpha} b(x) + c_{n,\alpha} \int \frac{\brac{\lapms{\alpha}
a(y)-\lapms{\alpha} a(x)}\ \brac{\lapms{\alpha} b(x)-\lapms{\alpha}
b(y)}}{\abs{x-y}^{n+\alpha}}\ dy + b(x)\ \lapms{\alpha} a(x).
\end{ma}
\]
Consequently,
\[
\begin{ma}
 &&\abs{H(\lapms{\alpha}a,\lapms{\alpha}b)(x)}\\
&\aleq& \intl_{\R^n}\intl_{\R^n}\intl_{\R^n}
\frac{\abs{\abs{y-\eta}^{-n+\alpha}-\abs{x-\eta}^{-n+\alpha}}\
\abs{\abs{y-\xi}^{-n+\alpha}-\abs{x-\xi}^{-n+\alpha}}}{\abs{x-y}^{n+\alpha}}\
\abs{a}(\eta)\ \abs{b}(\xi)\ dy\ d\xi\ d\eta.
\end{ma}
\]
Pick $\varepsilon \in (0,1)$ in Proposition~\ref{pr:multest} such that
\[
 \alpha < 2\varepsilon < \min\{2\alpha,n+\alpha\}.
\]
Let us see, for instance, how our expression behaves if the ``Type I``-estimate from the proof of Proposition~\ref{pr:multest} is applicable, that is, if
\[
\frac{\abs{\abs{y-\eta}^{-n+\alpha}-\abs{x-\eta}^{-n+\alpha}}\
\abs{\abs{y-\xi}^{-n+\alpha}-\abs{x-\xi}^{-n+\alpha}}}{\abs{x-y}^{n+\alpha}}\
\aleq {\abs{y-\eta}^{-n+\alpha-\varepsilon}\
\abs{x-\xi}^{-n+\alpha-\varepsilon}} {\abs{x-y}^{-n+2\varepsilon-\alpha}}
\]
Observe, by the choice of $\varepsilon$, all the appearing kernels on the right-hand side of this estimate have the exponent $n-\sigma$ for some $\sigma > 0$. That is, all the appearing kernels correspond to the kernel of a Riesz potential $\lapms{\sigma}$, see \eqref{eq:lo:rieszpotentialdef}. Namely,
\[
\int \abs{y-\eta}^{-n+\alpha-\varepsilon}\ \abs{a}(\eta)\ d\eta \approx \lapms{\alpha - \varepsilon} \abs{a}(y),
\]
\[
\int \abs{x-\xi}^{-n+\alpha-\varepsilon}\ \abs{b}(\xi)\ d\xi \approx \lapms{\alpha-\varepsilon} \abs{b}(x),
\]
and finally
\[
\int {\abs{x-y}^{-n+2\varepsilon-\alpha}} \lapms{\alpha - \varepsilon} \abs{a}(y)\ dy \approx \lapms{2\varepsilon-\alpha}\lapms{\alpha - \varepsilon} \abs{a}(x) \approx \lapms{\varepsilon} \abs{a}(x).
\]
By these kind of arguments, one obtains
\[
\begin{ma}
 &&\abs{H(\lapms{\alpha}a,\lapms{\alpha}b)(x)}\\
&\aleq& \lapms{\varepsilon} \abs{a}(x)\ \lapms{\alpha-\varepsilon}\abs{b}(x) +
\lapms{\alpha-\varepsilon}\abs{a}(x)\ \lapms{\varepsilon} \abs{b}(x)\\
&&\lapms{2\varepsilon - \alpha}\brac{\lapms{\alpha-\varepsilon} \abs{a}\
\lapms{\alpha-\varepsilon}\abs{b}}(x)\\
&&+A,
\end{ma}
\]
where
\[
 \begin{ma}
  A &:=& \intl_{\R^n}\intl_{\R^n}\intl_{\R^n} \frac{\abs{x-\eta}^{-n+\alpha}
\abs{x-\xi}^{-n+\alpha}}{\abs{x-y}^{n+\alpha}}\ \chi_{\abs{x-y} > 2
\abs{x-\xi}}\ \chi_{\abs{x-y} > 2 \abs{x-\eta}}\ \abs{a}(\eta)\ \abs{b}(\xi)\
dy\ d\xi\ d\eta\\
&\aleq{}& \intl_{\R^n}\intl_{\R^n} \abs{x-\eta}^{-n+\alpha}\
\abs{x-\xi}^{-n+\alpha}\ \abs{x-\xi}^{-\frac{\alpha}{2}}\
\abs{x-\eta}^{-\frac{\alpha}{2}}\ \abs{a}(\eta)\ \abs{b}(\xi)\ d\xi\ d\eta\\
&=& \lap^{-\frac{\alpha}{4}}\abs{a}(x)\ \lap^{-\frac{\alpha}{4}}\abs{b}(x).
 \end{ma}
\]
Thus we can conclude the proof: we choose $L_\alpha = 3$, with
$s_1 = \varepsilon$, $t_1 = \alpha - \varepsilon$, $s_1-t_1 = 2\varepsilon -\alpha$ as first term, then with interchanged roles $s_2 = t_1$ and $t_2 = s_2$,  $s_2-t_2 = 2\varepsilon -\alpha$ and finally $s_3 = \alpha/2$, $t_3 = \alpha/2$\,.\\

\end{proof}
\begin{remark}
{\rm About the strategy of the above proof let us remark, how the decomposition argument we apply in \cite{Sfracenergy} and here, is related to the arguments in \cite{DR1dSphere} and \cite{SNHarmS10}: For the para-product estimates of the 3-term commutators in \cite{DR1dSphere}, the authors  used an infinite Taylor expansion after Fourier transform in the phase space whenever the symbols were rather close to each other.

Instead, in \cite{SNHarmS10} the mean value formula (that is: a one-step Taylor expansion) was employed, also after Fourier transform, which lead to a simple pointwise estimate in the phase space which sufficed for the purposes there, but did not have the full power of the more complicated, yet more generalizable argument in \cite{DR1dSphere}. Both of these arguments, essentially obtained \emph{pointwise} bounds in the \emph{phase space}, which -- transformed in the \emph{geometric space} -- gives not pointwise, but $L^p(\R^n)$-estimates.

With the method introduced in \cite{Sfracenergy} and employed here, we obtain pointwise results in the \emph{geometric space}.

But let us stress, that both, the arguments in \cite{SNHarmS10} in the phase space and in \cite{Sfracenergy} in the geometric space, which apply both a one-step Taylor expansion, are rough in the following sense: When considering Hardy-space or BMO-space estimates, they do not seem to give the optimal result (which is, anyways, not pointwise anymore), and in this case, the strategy of using para-products as developed in \cite{DR1dSphere} seems more viable, see, e.g., the Hardy-space estimates in \cite{DR1dSphere}.}
\end{remark}

\subsection{The case \texorpdfstring{$\alpha \geq 1$}{alpha >= 1}}
We will reduce the case $\alpha \geq 1$ to the case $\alpha \in (0,1)$ already
discussed.
For the case $\alpha = 1+{\tilde{\alpha}} \in [1,2)$, we use the following
argument: Let $\Rz_i$ be the $i$-th Riesz-transform,
\[
 \Rz_i f(x) = c\ \intl \frac{(x-y)^i}{\abs{x-y}^{n+1}} f(y)\ dy.
\]
Then for $1+{\tilde{\alpha}} \in [1,2)$,
\[
\begin{ma}
 \Rz_i H_{1+{\tilde{\alpha}}}(u,v) &=& \lapa \partial_i (uv) - \Rz_i (u
\laps{1+{\tilde{\alpha}}} v) - \Rz_i (v \laps{1+{\tilde{\alpha}}} u)\\
&=& \lapa (u \partial_i v) - u \laps{{\tilde{\alpha}}} \partial_i v -
\laps{{\tilde{\alpha}}} u \ \partial_i v \\
&& + \lapa (v \partial_i u) - v \laps{{\tilde{\alpha}}} \partial_i u -
\laps{{\tilde{\alpha}}} v \ \partial_i u \\
&&- \Rz_i (u \laps{1+{\tilde{\alpha}}} v) + u \laps{{\tilde{\alpha}}} \partial_i
v \\
&&- \Rz_i (v \laps{1+{\tilde{\alpha}}} u) + v \laps{{\tilde{\alpha}}} \partial_i
u \\
&& + \laps{{\tilde{\alpha}}} v \ \partial_i u + \laps{{\tilde{\alpha}}} u \
\partial_i v\\
&=& H_{{\tilde{\alpha}}} (u,\partial_i v) + H_{{\tilde{\alpha}}} (v,\partial_i
u)\\
&&- \Rz_i (u \laps{1+{\tilde{\alpha}}} v) + u \Rz_i \laps{1+{\tilde{\alpha}}} v
\\
&&- \Rz_i (v \laps{1+{\tilde{\alpha}}} u) + v \Rz_i \laps{1+{\tilde{\alpha}}} u
\\
&& + \laps{{\tilde{\alpha}}} v \ \partial_i u + \laps{{\tilde{\alpha}}} u \
\partial_i v.
\end{ma}
\]
In the case ${\tilde{\alpha}} = 0$, we write this as
\[
\begin{ma}
 \Rz_i H_{1}(u,v) &=&- \Rz_i (u \laps{1+{\tilde{\alpha}}} v) + u \Rz_i
\laps{1+{\tilde{\alpha}}} \partial_i v \\
&&- \Rz_i (v \laps{1+{\tilde{\alpha}}} u) + v \Rz_i\laps{1+{\tilde{\alpha}}} u
\\
\end{ma}
\]
The terms $H_{\tilde{\alpha}}(\cdot,\cdot)$, as we've seen above, and the terms
$\laps{{\tilde{\alpha}}} u \ \partial_i v$ and $\laps{{\tilde{\alpha}}} v \
\partial_i u$ are actually products of lower order operators applied to $u$ and
$v$, respectively. Moreover, we have the following estimate, which should be
compared to the famous commutator estimates and their relation to Hardy spaces
and BMO, by Coifman, Rochberg, Weiss \cite{CRW76} and the related work by
\cite{Chan82}.
\begin{proposition}
For ${\tilde{\alpha}} \in [0,1)$ there exist a constant $C_{\tilde{\alpha}} > 0$
and a number $L \equiv L_{\tilde{\alpha}} \in \N$, and for $k \in
\{1,\ldots,L\}$ constants $s_k \in (0,1)$, $t_k \in [0,s_k]$ such that
\[
 \Rz_i (G\ \lapms{1+{\tilde{\alpha}}} F) - (\Rz_i G\ \lapms{1+{\tilde{\alpha}}}
F) \aleq C\ \sum_{k=1}^L \lapms{s_k-t_k} \brac{\lapms{t_k} \abs{G}\
\lap^{-\frac{1+{\tilde{\alpha}}}{2}+\frac{s_k}{2}} \abs{F}}.
\]
Moreover, $\abs{s_k-t_k}$ can be supposed to be arbitrarily small.
\end{proposition}
\begin{proof}
As in the case of $H_\alpha$ for $\alpha \in (0,1)$ above, we exploit that the
difference of the involved operators can be expressed by the differences of
their kernels, more precisely
\[
 \begin{ma}
   &&\Rz_i (G\ \lapms{1+{\tilde{\alpha}}} F)(x) - (\Rz_i G\
\lapms{1+{\tilde{\alpha}}} F(x) \\
&=& c\int \int \frac{(x-y)^i}{\abs{x-y}^{n+1}}\
(\abs{y-z}^{-n+1+{\tilde{\alpha}}} - \abs{x-z}^{-n+1+{\tilde{\alpha}}}) \ G(y)\
F(z)\ dy\ dz\\
&\aleq&  c\int \int \frac{\abs{\abs{y-z}^{-n+1+{\tilde{\alpha}}} -
\abs{x-z}^{-n+1+{\tilde{\alpha}}}}}{\abs{x-y}^{n}} \ \abs{G}(y)\ \abs{F}(z)\ dy\
dz\\
 \end{ma}
\]
Using the mean value theorem, similar to the proof of
Proposition~\ref{pr:multest}, precisely as in \cite{Sfracenergy}, we conclude.
\end{proof}

In particular, we have
\begin{proposition}\label{pr:lowerorderalphal2}
Let $u =\lapms{\alpha} \laps{\alpha} u$, $v =\lapms{\alpha} \laps{\alpha} v$.
Then for $\alpha \in (0,2)$ there exists some constant $C_\alpha > 0$ and a
number $L \equiv L_\alpha \in \N$, and for $k \in \{1,\ldots,L\}$ constants $s_k
\in (0,\alpha)$, $t_k \in [0,s_k]$ such that for any $i = 1, \ldots,n$, where $\Rz_i$ denotes the Riesz transform,
\[
 \abs{\Rz_i H_\alpha(u,v)(x)} \leq C\ \sum_{k=1}^L \lapms{s_k-t_k}
\brac{M_k\lapms{t_k} \abs{\laps{\alpha} u}\ \N_k
\lap^{-\frac{\alpha}{2}+\frac{s_k}{2}} \abs{\laps{\alpha} v}}.
\]
Here, $M_k, N_k$ are possibly Riesz transforms, or the identity. Moreover,
$\abs{s_k-t_k}$ can be supposed to be arbitrarily small.
\end{proposition}

For $\alpha = K + \tilde{\alpha}$, $\tilde{\alpha} \in (0,2)$, $K \in \N$,
observe that
\[
\begin{ma}
 H_{\alpha}(u,v) &=& \laps{\tilde{\alpha}} (\lap^Ku\ v) + \laps{\tilde{\alpha}}
(u\ \lap^K v) + \sum_{\ontop{\abs{\gamma}+\sabs{\tilde{\gamma}} =
2K}{\abs{\gamma},\sabs{\tilde{\gamma}}\geq 1}} c_{\gamma,\tilde{\gamma}}\
\laps{\tilde{\alpha}} (\partial^\gamma u \partial^{\tilde{\gamma}} v)\\
&&- u\ \lap^K \laps{\tilde{\alpha}} v - u \lap^K \laps{\tilde{\alpha}} v\\
&=& H_{\tilde{\alpha}}(\lap^K u,v)+\lap^K u\ \laps{\tilde{\alpha}}
v+H_{\tilde{\alpha}}( u,\lap^Kv)+\lap^K v\ \laps{\tilde{\alpha}} u\\
&&+\sum_{\ontop{\abs{\gamma}+\sabs{\tilde{\gamma}} =
2K}{\abs{\gamma},\sabs{\tilde{\gamma}}\geq 1}} c_{\gamma,\tilde{\gamma}}\
H_{\tilde{\alpha}} (\partial^\gamma u, \partial^{\tilde{\gamma}} v)\\
&&+\sum_{\ontop{\abs{\gamma}+\sabs{\tilde{\gamma}} =
2K}{\abs{\gamma},\sabs{\tilde{\gamma}}\geq 1}} c_{\gamma,\tilde{\gamma}}\
\laps{\tilde{\alpha}}\partial^\gamma u\ \partial^{\tilde{\gamma}}
v+\partial^\gamma u\ \laps{\tilde{\alpha}}\partial^{\tilde{\gamma}} v
\end{ma}
\]
Using that all terms which are not of the form $H_\alpha$, are actually products
of lower order operators, one concludes

\begin{lemma}\label{la:lowerorderalphaln}
Let $u =\lapms{\alpha} \laps{\alpha} u$, $v =\lapms{\alpha} \laps{\alpha} v$.
Then for $\alpha \in (0,n)$ there exists some constant $C_\alpha > 0$ and a
number $L \equiv L_\alpha \in \N$, and for $k \in \{1,\ldots,L\}$ constants $s_k
\in (0,\alpha)$, $t_k \in [0,s_k]$ such that for any $i = 1, \ldots,n$,
\[
 \abs{\Rz_i H_\alpha(u,v)(x)} \leq C\ \sum_{k=1}^L M_k \lapms{s_k-t_k}
\brac{\lapms{t_k} \abs{\laps{\alpha} u}\ N_k
\lap^{-\frac{\alpha}{2}+\frac{s_k}{2}} \abs{\laps{\alpha} v}}.
\]
Here, $M_k, N_k$ are possibly Riesz transforms, or the identity. Moreover,
$\abs{s_k-t_k}$ can be supposed to be arbitrarily small.
\end{lemma}
In particular,
\begin{proposition}\label{pr:Hest}
Let $\alpha \in (0,n)$, $q, q_1, q_2 \in [1,\infty]$ such that
\[
 \fracm{q} = \fracm{q_1} + \fracm{q_2}.
\]
Then
\[
 \vrac{H_\alpha (u,v)}_{(\frac{n}{\alpha},q),\R^n} \aleq \vrac{\lapa
u}_{(\frac{n}{\alpha},q_2),\R^n}\ \vrac{\lapa v}_{(\frac{n}{\alpha},q_2),\R^n}.
\]
\end{proposition}

\subsection{Local Estimates}
In this section, the goal is give in Lemma~\ref{la:Hvwlocest} a localized
version of Proposition~\ref{pr:Hest}.


Similar to the arguments in Proposition \ref{pr:HwvpBlrest}, one can show
\begin{proposition}\label{pr:locestloop1}
There is $\gamma > 0$ such that for any $\Lambda > 4$, $s \in (0,\alpha)$, $t
\in [0,s]$,
\[
 \Vert \lapms{s-t} \brac{\lapms{t} \abs{\eta_{-\Lambda r } a}\
\lap^{-\frac{\alpha}{2}+\frac{s}{2}} \abs{b}} \Vert_{\frac{n}{\alpha},B_r} \aleq
\Lambda^{-\gamma}\ \Vert a \Vert_{\frac{n}{\alpha}}\ \Vert b
\Vert_{\frac{n}{\alpha}}
\]
and for $\Lambda_1,\Lambda_2 > 4$
\[
 \Vert \lapms{s-t} \brac{\lapms{t} \abs{\eta_{-\Lambda_1 r} a}\
\lap^{-\frac{\alpha}{2}+\frac{s}{2}} \brac{\eta_{-\Lambda_2 r} \abs{b}}}
\Vert_{\frac{n}{\alpha},B_r} \aleq \Lambda_1^{-\gamma}\ \ \Lambda_2^{-\gamma}\
\Vert a \Vert_{\frac{n}{\alpha}}\ \Vert b \Vert_{\frac{n}{\alpha}}
\]
\end{proposition}
\begin{proof}
For some $\psi \in C_0^\infty(B_r)$, $\Vert \psi \Vert_{\frac{n}{n-\alpha}} \leq
1$ we have to estimate (up to possibly Riesz transforms, which we will ignore
again, as they do not change the argument)
\[
\begin{ma}
 &&\int \lapms{s-t} \brac{\lapms{t} \abs{\eta_{-\Lambda} a}\
\lap^{-\frac{\alpha}{2}+\frac{s}{2}} \abs{b}}\ \psi\\
&=& \int \lapms{t} \abs{\eta_{-\Lambda} a}\ \lap^{-\frac{\alpha}{2}+\frac{s}{2}}
\abs{b}\ \lapms{s-t} \psi\\
&\aleq{}& \Vert \eta_{\sqrt{\Lambda}} \lapms{t} \abs{\eta_{-\Lambda} a}
\Vert_{\frac{n}{\alpha -t}}\
\Vert \lap^{-\frac{\alpha}{2}+\frac{s}{2}} \abs{b} \Vert_{\frac{n}{s}}\
\Vert \lapms{s-t} \psi \Vert_{\frac{n}{n-\alpha-s+t}}\\
&& +\Vert \lapms{t} \abs{\eta_{-\Lambda} a} \Vert_{\frac{n}{\alpha -t}}\
\Vert \lap^{-\frac{\alpha}{2}+\frac{s}{2}} \abs{b} \Vert_{\frac{n}{s}}\
\Vert (1-\eta_{\sqrt{\Lambda}}) \lapms{s-t} \psi
\Vert_{\frac{n}{n-\alpha-s+t}}\\
&\aleq{}& \Vert \eta_{\sqrt{\Lambda}} \lapms{t} \abs{\eta_{-\Lambda} a}
\Vert_{\frac{n}{\alpha -t}}\
\Vert b \Vert_{\frac{n}{\alpha}} +\Vert a \Vert_{\frac{n}{\alpha}}\
\Vert b \Vert_{\frac{n}{\alpha}}\
\Vert (1-\eta_{\sqrt{\Lambda}}) \lapms{s-t} \psi
\Vert_{\frac{n}{n-\alpha-s+t}}\\
\end{ma}
\]
Then the first claim follows the arguments for products of (non-local)
fractional operators with disjoint support, cf., e.g., the arguments of
Proposition~\ref{pr:stupidestumw} or \cite[Proposition 4.4]{Sfracenergy}.\\
The second claim follows by the same method, only taking more care in the cutoff
for $b$.
\end{proof}

\begin{lemma}\label{la:Hvwlocest}
Let $v, w \in W^{\alpha,p}(\R^n)$. Then for any $\Lambda > 2$
\[
\begin{ma}
\Vert H(v,w) \Vert_{\frac{n}{\alpha},B_r} &\aleq& \Vert \eta_{\Lambda r}
\laps{\alpha} v \Vert_{\frac{n}{\alpha}}\ \Vert \eta_{\Lambda r} \laps{\alpha} w
\Vert_{\frac{n}{\alpha}}\\
&& + \Lambda^{-\gamma}\ \Vert \laps{\alpha}w \Vert_{\frac{n}{\alpha}}\ \Vert
\eta_{\Lambda r} \laps{\alpha}v \Vert_{\frac{n}{\alpha}}\\
&& + \Lambda^{-\gamma}\ \Vert \laps{\alpha}v \Vert_{\frac{n}{\alpha}}\ \Vert
\eta_{\Lambda r} \laps{\alpha}w \Vert_{\frac{n}{\alpha}}\\
&& + \Lambda^{-\gamma}\ \Vert \laps{\alpha}v \Vert_{\frac{n}{\alpha}}\
\sum_{k=1}^\infty 2^{-k\gamma}\ \Vert \eta_{\Lambda r}^k \laps{\alpha} w
\Vert_{\frac{n}{\alpha}}.
\end{ma}
\]
\end{lemma}
\begin{proof}
Set
\[
 v_{\Lambda} := \lapms{\alpha} \brac{\eta_{\Lambda r} \laps{\alpha} v},
\]
\[
 v_{-\Lambda} := \lapms{\alpha} \brac{(1-\eta_{\Lambda r}) \laps{\alpha} v},
\]
and
\[
 v_{-\Lambda,k} := \lapms{\alpha} \brac{\eta^k_{\Lambda r} \laps{\alpha} v}.
\]
Then
\[
H(v,w) = H(v_{\Lambda},w_{\Lambda}) + H(v_{-\Lambda},w_{\Lambda}) +
H(v_{\Lambda},w_{-\Lambda}) + H(v_{-\Lambda},w_{-\Lambda}).
\]
We have,
\[
\Vert H(v_{\Lambda},w_{\Lambda}) \Vert_{\frac{n}{\alpha}} \aleq \Vert
\eta_{\Lambda r} \laps{\alpha} v \Vert_{\frac{n}{\alpha}}\ \Vert \eta_{\Lambda
r} \laps{\alpha} w \Vert_{\frac{n}{\alpha}}.
\]
Moreover, the second and the third term are controlled by means of
Proposition~\ref{pr:locestloop1}. It remains to estimate
$H(v_{-\Lambda},w_{-\Lambda})$,
\[
 H(v_{-\Lambda},w_{-\Lambda}) = \sum_{k=1}^\infty H(v_{-\Lambda},w_{-\Lambda,k})
\]
and again by Proposition~\ref{pr:locestloop1},
\[
  \Vert H(v_{-\Lambda},w_{-\Lambda,k}) \Vert_{\frac{n}{\alpha},B_r} \aleq
2^{-\gamma k} \Lambda^{-2\gamma}\ \Vert v \Vert_{\frac{n}{\alpha}}\ \Vert
\eta_{\Lambda r}^k \laps{\alpha} w \Vert_{\frac{n}{\alpha}}
\]
\end{proof}

\newpage
\renewcommand{\thesection}{B}
\renewcommand{\thesubsection}{B.\arabic{subsection}}

\section{Decomposition in Euclidean Spaces}
In the proof of Lemma \ref{la:wallest} we used the following fact, which permitted us to get a full information of $\vrac{\lapa w}_{\pbar}$ from the information of $\vrac{\omega_{ij} w^{j}\lapa w^i}_{\pbar}$, and the normal projection $\vrac{w^{i}\lapa w^i}_{\pbar}$.
\begin{proposition}\label{pr:orthogdecomp}
Let $p \in \R^N$, $\abs{p} = 1$. Set
\[
 \Omega := \{\omega \equiv (\omega_{ij})_{i,j =1}^N: \quad \omega_{ij} = -
\omega_{ij} \in \{-1,0,1\}  \},
\]
and let for $\omega \in \Omega$
\[
 p_\omega \equiv ({p_\omega}_i)_{i=1}^N := (\omega_{ij}p^j)_{i=1}^N.
\]
Then, for uniform constants $c,C$, depending only on the dimension $N$,
\begin{equation}\label{eq:antisym:absqcontrol}
 c\ \abs{q} \leq \sum_{\omega \in \Omega} \abs{\langle p_\omega, q\rangle} +
\abs{\langle p, q\rangle} \leq C\  \abs{q} \quad \mbox{for all $q \in \R^N$.}
\end{equation}
In particular, there is a uniform $\theta \in (0,1)$ such that for any $q \in
\R^N$, if for some $\omega \in \Omega$,
\[
 \abs{\langle p_\omega, q\rangle} \leq \theta \abs{q},
\]
then there exists $\tilde{\omega} \in \Omega$, such that
\[
 \abs{\langle p_\omega, q\rangle} \leq \theta \abs{q} \leq \max \left
\{\abs{\langle p_{\tilde{\omega}}, q\rangle},\  \abs{\langle p,q\rangle}\right
\}.
\]
\end{proposition}
For the convenience of the reader, we will give a proof:
\begin{proof}
We prove the first inequality of the claim \eqref{eq:antisym:absqcontrol} for $\abs{q} = 1$ and $q \perp p$. The case $N =1$ is trivial, of course, so assume $N \geq 2$. Let for $\alpha \neq \beta$,
\[
\omega^{\alpha\beta}_{i,j} := \delta_{i}^\alpha \delta_{j}^\beta
-\delta_{j}^\alpha \delta_{i}^\beta \in \Omega,
\]
where $\delta$ denotes the Kronecker symbol
\[
 \delta_\alpha^i = \begin{cases}
		    1 \quad & \mbox{if $\alpha = i$,}\\
                    0 \quad & \mbox{else.} 
                   \end{cases}
\]
The claim \eqref{eq:antisym:absqcontrol} then follows, if we can show that there exists a uniform constant $c > 0$ such that for any $p \equiv (p_i)_{i=1}^N$, $q \equiv (q_i)_{i=1}^N \in \R^N$ with $\abs{p} = \abs{q} = 1$, and $p \perp q \in \R^N$, there are
$\alpha \neq \beta$ such that
\[
  \abs{  \langle p_{\omega^{\alpha\beta}},q\rangle_{\R^N}} \equiv \sabs{ p_\alpha q_\beta -p_\beta q_\alpha} \geq c.
\]
Choose $i_0 \neq k_0$ such that (w.l.o.g) $p_{i_0}, q_{k_0} > 0$ and, more importantly,
\begin{equation}\label{eq:antis:li0mk0geqc0}
 p_{i_0}, q_{k_0} \geq \frac{1}{2\sqrt{N-1}} =: c_0.
\end{equation}
In fact, by orthonormality of $p \perp q$, it cannot happen, that for some $i$, both,
$\abs{p_i}^2, \abs{q_i}^2 \geq \frac{3}{4}$. Thus, if there is some $i_0$ such that $\abs{p_{i_0}}^2 \geq \frac{3}{4}$, then $q_{i_0}^2 \leq \frac{3}{4}$ and because of $\abs{q} = 1$ there has to exist at least one $k_0 \neq i_0$ such that $q_{k_0}$ satisfies \eqref{eq:antis:li0mk0geqc0}. An analogous argument holds, if $\abs{q_{k_0}}^2 \geq \frac{3}{4}$ for some $k_0$. In the remaining cases, we can assume that $\abs{p_i}^2, \abs{q_{k}}^2 < \frac{3}{4}$ for all $i,k$. But then there have to be at least two indices $i_0 \neq i_1$ such that $p_i$ satisfies \eqref{eq:antis:li0mk0geqc0} for $i= i_1, i_1$: because if this was not the case, we had the following contradiction:
\[
 1 = \sum_{i=1}^N \abs{p_i}^2 < \frac{3}{4} + (N-1)\frac{1}{4(N-1)} = 1.
\]
We thus have two different choices for both for $i_0$ in order for $p_{i_0}$ to satisfy \eqref{eq:antis:li0mk0geqc0}, so whichever the choice of $k_0$ is, we can choose $i_0 \neq k_0$.
Then,
\[
 \langle p_{\omega^{i_0k_0}}, q \rangle_{\R^N} = p_{i_0} q_{k_0} - p_{k_0}
q_{i_0}.
\]
If $p_{k_0} q_{i_0} \leq 0$, this implies,
\[
  \langle p_{\omega^{i_0k_0}}, q \rangle_{\R^N}  \geq \brac{c_0}^2.
\]
Assume on the other hand, that $p_{k_0} q_{i_0} > 0$ and that even
\[
 \langle p_{\omega^{i_0k_0}}, q \rangle  \leq \frac{1}{10} \brac{c_0}^2.
\]
Then,
\[
 \abs{p_{k_0} q_{i_0}} = p_{k_0} q_{i_0} = p_{i_0} q_{k_0}-\langle p_{\omega^{i_0k_0}}, q \rangle_{\R^N}  \geq (c_0)^2 -\fracm{10}(c_0)^2.
\]
Since $\abs{p_{k_0}},\abs{q_{i_0}} \leq 1$, we infer
\[
 \abs{p_{k_0}}, \abs{q_{i_0}} \geq \frac{9}{10}(c_0)^2.
\]
It follows, that

\begin{equation}\label{eq:antisym:c1}
 c_1 := \frac{9}{10}(c_0)^3 \leq \abs{q_{i_0}} \abs{p_{i_0}} + \abs{q_{k_0}}
\abs{p_{k_0}}.
\end{equation}
Because of orthogonality $p \perp q$,
\begin{equation}\label{eq:antisym:lastineq}
 - \brac{q_{i_0} p_{i_0} + q_{k_0} p_{k_0}} =  \sum_{i \neq
i_0,k_0} p_i q_i.
\end{equation}
As the product $p_{k_0} q_{i_0} > 0$ we have to consider the case $I$, where both, 
$p_{k_0}, q_{i_0} > 0$ and the case $II$ where both, $p_{k_0}, q_{i_0}
< 0$. Recall that in both cases $p_{i_0}, q_{k_0} > 0$. As for case $I$, \eqref{eq:antisym:lastineq}, \eqref{eq:antisym:c1} imply
\[
 -c_1 \geq \sum_{\ontop{i \neq i_0,k_0}{p_i q_i < 0}} p_i q_i +
\sum_{\ontop{i \neq i_0,k_0}{p_i q_i > 0}} p_i q_i =
-\sum_{\ontop{i \neq i_0,k_0}{p_i q_i < 0}} \abs{p_i} \abs{q_i}
+ \sum_{\ontop{i \neq i_0,k_0}{p_i q_i > 0}} \abs{p_i} \abs{q_i}
\geq -\sum_{\ontop{i \neq i_0,k_0}{p_i q_i < 0}} \abs{p_i}
\abs{q_i}.
\]
In particular, this implies, that there is $i_1 \neq i_0,k_0$ such that
$p_{i_1} q_{i_1} < 0$ and still (recall $\abs{p_{i_1}}, \abs{q_{i_1}} \leq 1$)
\[
 \abs{p_{i_1}}, \abs{q_{i_1}} \geq \frac{c_1}{N}.
\]
Assuming $p_{i_1} > 0$, $q_{i_1} < 0$, since we know $q_{i_0} > 0$, $p_{i_0} > 0$
\[
 \langle p_{\omega^{i_1 i_0}} , q \rangle_{\R^N} = p_{i_1} q_{i_0} - p_{i_0}
q_{i_1} > \abs{p_{i_0}}\abs{q_{i_1}} \geq c_0\ \frac{c_1}{N}.
\]
If we assume, on the other hand, $p_{i_1} > 0$, $q_{i_1} < 0$,
\[
 -\langle p_{\omega^{i_1 i_0}} , q \rangle_{\R^N} = - p_{i_1} q_{i_0} + p_{i_0}
q_{i_1} < -\abs{p_{i_0}}\abs{q_{i_1}} \leq -c_0\ \frac{c_1}{N}.
\]
In the case $II$, where both $p_{k_0}, q_{i_0} < 0$, we have instead
\[
  c_1 \leq -\sum_{\ontop{i \neq i_0,k_0}{p_i q_i < 0}} \abs{p_i}
\abs{q_i} + \sum_{\ontop{i \neq i_0,k_0}{p_i q_i > 0}} \abs{p_i}
\abs{q_i} \leq \sum_{\ontop{i \neq i_0,k_0}{p_i q_i > 0}}
\abs{p_i} \abs{q_i},
\]
and we can find $i_2 \neq i_0,k_0$ such that $p_{i_2} q_{i_2} > 0$ and
\[
 \abs{p_{i_2}}, \abs{q_{i_2}} \geq \frac{c_1}{N}.
\]
Then,
\[
-\operatorname{sign} (p_{i_2})\ \langle p_{\omega^{{i_2 i_0}}} , q \rangle_{\R^N} =
\abs{p_{i_2}} \abs{q_{i_0}} + p_{i_0} \abs{q_{i_2}} \geq c_0 \frac{c_1}{N}.
\]
This proves the first claim, which essentially is contained in the following: For $\abs{p} = 1$, the finite set $\{p_\omega, \omega \in \Omega\}$ may not be a basis; Nevertheless, it is a linear generator of the space $p^\perp \subset \R^N$.

For the second claim we use the following argument: Since $\Omega$ is finite, any vector $q$ of length $\abs{q}=1$ has to have length at least $\theta = \frac{1}{\abs{\Omega}+1}$ in at least one of the linear spaces generated by some $p_\omega$ or generated by $p$ itself. That is $\abs{\langle p_\omega,q\rangle} \geq \theta$ for some $\omega \in \Omega$, or $\abs{\langle p,q\rangle } \geq \theta$. Consequently, if there is some $\omega$ such that $\abs{\langle p_\omega,q\rangle }  < \theta$, then there has to be another $\tilde{\omega} \in \Omega$ such that
$\abs{\langle p_{\tilde{\omega}},q\rangle } \geq \theta$ or, alternatively, $\abs{\langle p,q\rangle }  \geq \theta$.

\end{proof}

\newpage
\bibliographystyle{alpha}%
\bibliography{bib}%

\begin{thebibliography}{DLR11b}

\bibitem[Cha82]{Chan82}
S.~Chanillo.
\newblock A note on commutators.
\newblock {\em Indiana Univ. Math. J.}, 31(1):7--16, 1982.

\bibitem[CRW76]{CRW76}
R.~R. Coifman, R.~Rochberg, and G.~Weiss.
\newblock Factorization theorems for {H}ardy spaces in several variables.
\newblock {\em Ann. of Math. (2)}, 103(3):611--635, 1976.

\bibitem[DL10]{DndMan}
F.~Da~Lio.
\newblock {Fractional Harmonic Maps into Manifolds in odd dimension $n>1$}.
\newblock {\em Preprint, arXiv:1012.2741}, 2010.

\bibitem[DLR11a]{DR1dMan}
F.~Da~Lio and T.~Rivi{\`e}re.
\newblock {Sub-criticality of non-local Schrödinger systems with antisymmetric
  potentials and applications to half-harmonic maps}.
\newblock {\em Advances in Mathematics}, 227(3):1300 -- 1348, 2011.

\bibitem[DLR11b]{DR1dSphere}
F.~Da~Lio and T.~Rivi{\`e}re.
\newblock Three-term commutator estimates and the regularity of 1/2-harmonic
  maps into spheres.
\newblock {\em Analysis and PDE}, 4(1):149 -- 190, 2011.

\bibitem[DLS12]{DSpalphMAN}
F.~Da~Lio and A.~Schikorra.
\newblock {($n$,$p$)-harmonic maps into manifolds, The case $p < 2$}.
\newblock {\em in preparation}, 2012.

\bibitem[DM10]{DM10}
F.~Duzaar and G.~Mingione.
\newblock Local lipschitz regularity for degenerate elliptic systems.
\newblock {\em Ann.I.H.P.(C)}, 27(6):1361 -- 1396, 2010.

\bibitem[Fre73]{Frehse73}
J.~Frehse.
\newblock A discontinuous solution of a mildly nonlinear elliptic system.
\newblock {\em Math. Z.}, 134:229--230, 1973.

\bibitem[H{\'{e}}l90]{Hel90}
F.~H{\'{e}}lein.
\newblock {R\'{e}gularit\'{e} des applications faiblement harmoniques entre une
  surface et une sph\`{e}re}.
\newblock {\em C.R. Acad. Sci. Paris 311, S\'{e}rie I}, pages 519--524, 1990.

\bibitem[Hof98]{Hof98}
S.~Hofmann.
\newblock An off-diagonal {$T1$} theorem and applications.
\newblock {\em J. Funct. Anal.}, 160(2):581--622, 1998.

\bibitem[Kol10]{Kol10}
S.~Kolasi\'{n}ski.
\newblock {Regularity of weak solutions of $n$-dimensional $H$-Systems}.
\newblock {\em Differential and Integral Equations}, pages 1073--1090, 2010.

\bibitem[KP88]{KP88}
T.~Kato and G.~Ponce.
\newblock Commutator estimates and the {E}uler and {N}avier-{S}tokes equations.
\newblock {\em Comm. Pure Appl. Math.}, 41(7):891--907, 1988.

\bibitem[Riv07]{Riv06}
T.~Rivi{\`e}re.
\newblock Conservation laws for conformally invariant variational problems.
\newblock {\em Invent. Math.}, 168(1):1--22, 2007.

\bibitem[Riv08]{Riv08}
T.~Rivi{\`e}re.
\newblock The role of integrability by compensation in conformal geometric
  analysis.
\newblock {\em to appear in: Analytic aspects of problems from Riemannian
  Geometry S.M.F.}, 2008.

\bibitem[Sch10a]{SnHsystemOrlicz}
A.~Schikorra.
\newblock {A Note on Regularity for the n-dimensional H-System assuming
  logarithmic higher Integrability}.
\newblock {\em Preprint, arXiv:1012.2737}, 2010.

\bibitem[Sch10b]{SNHarmS10Arxiv}
A.~Schikorra.
\newblock {Regularity of n/2-harmonic maps into spheres}.
\newblock {\em Preprint, arXiv:1003.0646v1}, 2010.

\bibitem[Sch11]{Sfracenergy}
A.~Schikorra.
\newblock {Interior and Boundary-Regularity of Fractional Harmonic Maps via
  Helein's Direct Method}.
\newblock {\em Preprint, arXiv:1103.5203}, 2011.

\bibitem[Sch12]{SNHarmS10}
A.~Schikorra.
\newblock {Regularity of n/2-harmonic maps into spheres}.
\newblock {\em J. Differential Equations}, 252:1862--1911, 2012.

\bibitem[Str94]{Strz94}
P.~Strzelecki.
\newblock Regularity of {$p$}-harmonic maps from the {$p$}-dimensional ball
  into a sphere.
\newblock {\em Manuscripta Math.}, 82(3-4):407--415, 1994.

\bibitem[Wan05]{WangCompThm05}
C.~Wang.
\newblock A compactness theorem of n-harmonic maps.
\newblock {\em Annales de l'Institut Henri Poincare (C) Non Linear Analysis},
  22(4):509 -- 519, 2005.

\end{thebibliography}
\end{document}